\documentclass[11pt]{article}
\usepackage{indentfirst}
\usepackage{latexsym}
\usepackage{hyperref}
\usepackage{graphicx}
\usepackage{mathrsfs}
\usepackage{bookmark}
\usepackage{enumerate}
\usepackage{amsfonts,amsthm,amssymb,mathtools}
\usepackage{amsmath}
\usepackage{latexsym, euscript, epic, eepic, color}
\usepackage{amssymb}

\usepackage{enumerate}
\usepackage[all]{xy}
\usepackage{setspace}
\allowdisplaybreaks

\usepackage{epstopdf}
\numberwithin{equation}{section}
\newtheorem{theorem}{Theorem}[section]

\newtheorem{definition}[theorem]{Definition}

\newtheorem{lemma}[theorem]{Lemma}

 \allowdisplaybreaks

\begin{document}
\baselineskip=16pt

\title{Signless Laplacian spectral radius and matchings in graphs\footnote{This work was supported by the National Natural Science Foundation of China [61773020]. }
}

\author{Chang Liu, Yingui Pan, Jianping Li\thanks{Corresponding author:Jianping Li (lijianping65@nudt.edu.cn)}\\
	\small  College of Liberal Arts and Sciences, National University of Defense Technology, \\
	\small  Changsha, China, 410073.\\
}
\date{\today}

\maketitle

\begin{abstract}
The signless Laplacian matrix $Q(G)$ of a graph $G$ is defined as  $Q(G)=D(G)+A(G)$, where $D(G)$ is the diagonal matrix of vertex degrees and $A(G)$ is the adjacency matrix of $G$. Let $q_1(G)$ denote the signless Laplacian spectral radius of $G$, i.e. the largest eigenvalue of the signless  Laplacian matrix $Q(G)$. Let $r(n)$ be the largest root of the equation $x^3-(3n-7)x^2+n(2n-7)x-2(n^2-7n+12)=0$. In this paper, we prove that for a positive even integer $n\geq4$, if $G$ is an $n$-vertex connected graph with \begin{equation*}
		q_1(G)>\begin{cases} $r(n)$, & \mbox{for }n\geq10\mbox{ or } n=4,\\ 
			4+2\sqrt{3}, & \mbox{for }n=6,\\
			6+2\sqrt{6}, & \mbox{for }n=8, \end{cases}
\end{equation*}
then $G$ has a perfect matching. It is sharp in the sense that there exist graphs $H$ such that $H$ has no perfect matching and $q_1(H)$ equals the lower bound for $q_1(G)$ for every positive even integer $n\geq4$.
\\[2pt]
\end{abstract}

\section{Introduction}
Let $G$ be a simple connected graph with vertex set $V(G)$ and edge set $E(G)$. The signless Laplacian matrix $Q(G)$ of $G$ is defined as $Q(G)=D(G)+A(G)$, where $D(G)$ is the diagonal matrix of vertex degrees and $A(G)$ is the adjacency matrix of $G$. The largest eigenvalue of $Q(G)$, written as  $q_1(G)$, is called the signless Laplacian spectral radius of $G$. For a vertex subset $S\subset V(G)$, let $G[S]$ be the subgraph of $G$ induced by the vertex set $S$, and let $G-S$ be the graph obtained from $G$ by deleting the vertices in $S$ together with their incident edges. Let $G_1\vee G_2$ denote the join of two graphs $G_1$ and $G_2$, which is the graph such that $V(G_1\vee G_2)=V(G_1)\cup V(G_2)$ and $E(G_1\vee G_2)=E(G_1)\cup E(G_2)\cup \left\lbrace uv|u\in V(G_1), v\in V(G_2) \right\rbrace $.


A \textit{matching} in a graph  is a set of disjoint edges, and a \textit{perfect matching} in $G$ is a matching covering all vertices of $G$. The sufficient and necessary condition for the existence of perfect matchings in a graph is first given by the Tutte's 1-factor theorem \cite{Tut}, which states $G$ has a perfect matching if and only if $o(G-S)-|S|\leq 0$ for any $S\subset V(G)$, where $o(G-S)$ is the number of odd components in $G-S$. In the past two decades, many researchers devoted to study the matchings in graphs, including maximum matching \cite{Cran,Cioaba,Sj,Jv}, fractional matching \cite{rebe,Yl1,Yl2,SO1}, rainbow matching \cite{LeS}, and matching energy \cite{Chen,Ji} and so on. Among these researches, relations between the eigenvalues and the matchings in graphs have been a hotpot. Chang \cite{chang1} investigated the largest eigenvalue of trees with perfect matchings, while Chang and Tian \cite{chang2} investigated the largest eigenvalue of unicyclic graphs with perfect matchings. In 2005, Brouwer and Haemers \cite{aebro} found some sufficient conditions for the existence of perfect matchings in a graph in terms of the Laplacian matrix spectrum, and they also gave an improved result for an $r$-regular graph in terms of the third
largest adjacency eigenvalue. Later, Cioab\v{a}  et al. \cite{smcio2} gave a best upper bound on the third largest eigenvalue to ensure that an $r$-regular graph $G$ with order $n$ has a perfect matching when $n$ is even, and a matching of order $n-1$ when $n$ is odd.

In 2016, O \cite{SO4} investigated the existence of fractional perfect matchings in terms of the spectral radius of graphs. Later, Pan et al. \cite{Yp} furtherly studied this topic from the signless Laplacian spectral radius of graphs.
Very recently,  O \cite{SO3} obtained the relation between the spectral radius and perfect matchings in an $n$-vertex graph $G$ by firstly determining a lower bound on the number of edges of $G$ which guarantees the existence of a perfect matching in $G$.

%
\begin{theorem}[See Theorem 1.2 in \cite{SO3}]\label{edgematch}
	Let $G$ be an $n$-vertex connected graph, where $n\geq4$ is an even number. Then $G$ has a perfect matching if
	\begin{equation*}
		|E(G)|>\begin{cases} \frac{1}{2}n^2-\frac{5}{2}n+5, & \mbox{for }n\geq10\mbox{ or } n=4,\\ 
			9, & \mbox{for }n=6,\\
			18, & \mbox{for }n=8. \end{cases}
	\end{equation*}
\end{theorem}

Moreover, O pointed out that this bound is sharp by calculating the number of edges of graphs $K_{n-3}\vee K_1 \vee \overline{K_2}$, $K_2\vee\overline{K_4}$ and $K_3\vee\overline{K_5}$ , which has no perfect matchings. 

Motivated by \cite{SO4,SO3,Yp}, it is natural to consider the relations between the signless Laplacian spectral radius and  matchings in graphs. By following the proof of O \cite{SO3}, we obtain the main results as below.

 \begin{theorem}\label{them1}
	Let $G$ be an $n$-vertex connected graph, where $n\geq4$ is an even number.  The largest root of the equation $x^3-(3n-7)x^2+n(2n-7)x-2(n^2-7n+12)=0$ is denoted by $r(n)$. Then $G$ has a perfect matching if
	\begin{equation*}
	q_1(G)>\begin{cases} $r(n)$, & \mbox{for }n\geq10\mbox{ or } n=4,\\ 
	4+2\sqrt{3}, & \mbox{for }n=6,\\
	6+2\sqrt{6}, & \mbox{for }n=8. \end{cases}
	\end{equation*}
\end{theorem}

Theorem \ref{edgematch} and \ref{them1} show that if $G$ has the maximum number of edges among the $n$-vertex graphs without a perfect matching, then $G$ has the maximum signless Laplacian spectral radius among them, except $n=6$ or $n=8$. The two special cases $n=6$ and $n=8$ says that even if graphs $G$ and $H$ share a certain graph property and $|E(G)|>|E(H)|$, we cannot guarantee that $q_1(G)>q_1(H)$. 

\section{Preliminaries}

For a matrix $B$, let $\rho(B)$ be the largest eigenvalue of $B$.
\begin{lemma}\cite{Cgod}\label{largerrd}
	Let $B$ and $B_1$ be real nonnegative matrices such that $B-B_1$ is nonnegative, then $\rho(B_1)\leq\rho(B)$.
\end{lemma}

We now explain the concepts of equitable matrices and equitable partitions.
\begin{definition}\label{def1}\cite{aebrouw}
	Suppose $B$ is a symmetric real matrix of order $n$ whose rows and columns are indexed by $P=\left\lbrace1,2,\cdots,n \right\rbrace $. Let $\left\lbrace P_1,P_2,\cdots,P_m \right\rbrace $ be a partition of $P$. Denote $n_i=|P_i|$ and then $n=n_1+n_2+\cdots+n_m$. Let $B$ be partitioned according $\left\lbrace P_1,P_2,\cdots,P_m \right\rbrace$, that is
	\begin{equation*}
	B=\begin{pmatrix}
	B_{1,1} & B_{1,2} & \cdots & B_{1,m}\\
	B_{2,1} & B_{2,2} & \cdots & B_{2,m}\\
	\vdots & \vdots & \ddots & \vdots\\
	B_{m,1} & B_{m,2} & \cdots & B_{m,m}
	\end{pmatrix}_{n\times n},
	\end{equation*}
	where the blocks $B_{i,j}$ denotes the submatrix of $B$ formed by rows in $P_i$ and the $P_j$ columns. Let $c_{i,j}$ denote the average row sum of $B_{i,j}$. Then the matrix $C=(c_{i,j})$ is called the quotient matrix of $B$ w.r.t. the given partition. Particularly, if the row sum of each submatrix $B_{i,j}$ is constant then the partition is called equitable.
\end{definition}

\begin{lemma}\cite{lhYou}\label{equit}
	Let $C$ be an equitable quotient matrix of $B$ as defined in Definition \ref{def1}. If $B$ is a nonnegative matrix, then $\rho(C)=\rho(B)$.
\end{lemma}


\section{Proof of Theorem \ref{them1}}

%
Similar to the proof technique used by O \cite{SO3}, we now give the proof of Theorem \ref{them1}.
\begin{proof}
	Assume to the contrary that $G$ has no perfect matching.  By Tutte's 1-factor theorem, there exists $S\subset V(G)$ such that $o(G-S)-|S|\geq1$, and all components of $G-S$ are odd, otherwise, we can remove one vertex from each even component to the set S, in consequence, the number of odd component and the number of vertices in $S$ have the same increase, so that
	$o(G-S)$ is always larger than $|S|$ and all components of $G-S$ are odd. Since $n$ is an even positive integer, then we have $o(G-S)$ and $|S|$ have same parity. Let $k=o(G-S)$ and $s=|S|$, then $k\geq s+2$. 
	
	Let $G_1, G_2,\cdots,G_k$ be the components of $G-S$ with $|V(G_1)|\geq |V(G_2)|\geq\cdots\geq|V(G_k)|$. Note that $n=s+n_1+n_2+\cdots+n_k$, where $n_i=|V(G_i)|$ ($i=1,2,\cdots,k$). To find the feasible maximum signless Laplacian spectral radius, we construct a new graph $G'$ by joining $S$ and $G-S$ and by adding edges in $S$ and in all components in $G-S$ so that $G[S]$ and all components in $G'-S$ are cliques. By Lemma \ref{largerrd}, we get $q_1(G)\leq q_1(G')$.
	
	
    The quotient matrix of the signless Laplacian matrix $Q(G')$ of the graph $G'$ with the vertex partition $\left\lbrace S,V(G_1),V(G_2),\cdots,V(G_k)\right\rbrace $ can be expressed as
	\begin{equation*}
	M_1 =\begin{pmatrix}
		n+s-2  & n_1      & n_2      & \cdots & n_k      \\
		s      & 2n_1+s-2 & 0        & \cdots & 0        \\
		s      & 0        & 2n_2+s-2 & \cdots & 0        \\
		\vdots & \vdots   & \vdots   &        & \vdots   \\
		s      & 0        & 0        & \cdots & 2n_k+s-2
	\end{pmatrix}
	\end{equation*}
	
	Let $f(x)$ be the characteristic polynomial of the matrix $M_1$. By expanding the above matrix determinant on the first line, we have
	\begin{align}
		f(x)&=(x-n-s+2)(x-2n_1-s+2)\cdots(x-2n_k-s+2)-sn_1(x-2n_2-s+2)\cdots(x-2n_k-s+2)\nonumber\\
		&~~~~~~+sn_2(x-2n_1-s+2)(x-2n_3-s+2)\cdots(x-2n_k-s+2)+\cdots\nonumber\\
		&~~~~~~+(-1)^{i}sn_i(x-2n_1-s+2)\cdots(x-2n_{i-1}-s+2)(x-2n_{i+1}-s+2)\cdots(x-2n_k-s+2)+\cdots\nonumber\\
		&~~~~~~+(-1)^{k}sn_k(x-2n_1-s+2)\cdots(x-2n_{k-1}-s+2)\label{poly1}.
	\end{align}
	
	Note that $M_1$ is an equitable quotient matrix of $Q(G')$. By Lemma \ref{equit}, we obtain $q_1(G')=r_{f}$, where $r_f$ is the largest root of the equation $f(x)=0$. Moreover, by Lemma \ref{largerrd}, we have $r_f>n+s-2$, and $r_f>2n_1+2s-2$.
	
	
	If $n_k\geq3$, similarly, we consider a new graph $G''$, which is obtained from $G'$
	by deleting two vertices in $G_k$ and adding two vertices to $G_1$ by joining the two vertices to the vertices in $V(G_1)$ and $S$.
For the partition $\left\lbrace S,V(G'_1),V(G_2),\cdots,V(G_{k-1}),V(G'_k)\right\rbrace$ of  $G''$, the corresponding quotient matrix $M_2$ of $Q(G'')$ has the characteristic polynomial $\tilde{f}(x)$ obtained from $f(x)$ by replacing $n_1$ and $n_k$ by $n_1+2$ and $n_k-2$, respectively. Thus,
	\begin{align*}
	\tilde{f}(x)&=(x-n-s+2)(x-2n_1-s-2)(x-2n_2-s+2)\cdots(x-2n_k-s+6)\\
	&~~~~~~-s(n_1+2)(x-2n_2-s+2)\cdots(x-2n_k-s+6)\\
	&~~~~~~+sn_2(x-2n_1-s-2)(x-2n_3-s+2)\cdots(x-2n_k-s+6)+\cdots\\
	&~~~~~~+(-1)^{i}sn_i(x-2n_1-s-2)\cdots(x-2n_{i-1}-s+2)(x-2n_{i+1}-s+2)\cdots(x-2n_k-s+6)+\cdots\nonumber\\
	&~~~~~~+(-1)^{k}s(n_k-2)(x-2n_1-s-2)(x-2n_2-s+2)\cdots(x-2n_{k-1}-s+2)\\
	&=f(x)+8(n_k-n_1-2)(x-n-s+2)(x-2n_2-s+2)\cdots(x-2n_{k-1}-s+2)\\
	&~~~~~~-2(x+2n_1-2n_k-s+6)s(x-2n_2-s+2)\cdots(x-2n_{k-1}-s+2)\\
	&~~~~~~+8(n_k-n_1-2)sn_2(x-2n_3-s+2)\cdots(x-2n_{k-1}-s+2)\\
	&~~~~~~-8(n_k-n_1-2)sn_3(x-2n_2-s+2)(x-2n_4-s+2)\cdots(x-2n_{k-1}-s+2)+\cdots\\
	&~~~~~~+(-1)^{q}\cdot(-2)(x+2n_1-2n_k-s-2)s(x-2n_2-s+2)\cdots(x-2n_{k-1}-s+2).
	\end{align*}
	 
	 Note that $f(r_f)=0$, $r_f>n+s-2$, $r_f>2n_1+2s-2$, and $n_1\geq\cdots\geq n_k$. By plugging the value $r_f$ into $x$ of $\tilde{f}(x)$, we have
	 \begin{align*}
	 \tilde{f}(r_f)&=8(n_k-n_1-2)(r_f-n-s+2)(r_f-2n_2-s+2)\cdots(r_f-2n_{k-1}-s+2)\\
	 &~~~~~~-2(r_f+2n_1-2n_k-s+6)s(r_f-2n_2-s+2)\cdots(r_f-2n_{k-1}-s+2)\\
	 &~~~~~~+8(n_k-n_1-2)sn_2(r_f-2n_3-s+2)\cdots(r_f-2n_{k-1}-s+2)\\
	 &~~~~~~-8(n_k-n_1-2)sn_3(r_f-2n_2-s+2)(r_f-2n_4-s+2)\cdots(r_f-2n_{k-1}-s+2)+\cdots\\
	 &~~~~~~+(-1)^{q}\cdot(-2)(r_f+2n_1-2n_k-s-2)s(r_f-2n_2-s+2)\cdots(r_f-2n_{k-1}-s+2)<0,
	 \end{align*}
	 which implies that $r_{f}<r_{\tilde{f}}$, where $r_{\tilde{f}}$ is the largest root of the equation $\tilde{f}(x)=0$. Since $M_2$ is an equitable quotient matrix of $Q(G'')$, we have $q_1(G')<q_1(G'')$. From the above discussion, we have $G'$ has the maximum spectral radius if $n_1=n-s-k+1$ and $n_i=1$ for all $2\leq i\leq k$.  Now, we suppose $|E(G')|=\dbinom{s+n_1}{2}+s(k-1)$, where $n_1=n-s-k+1$. There exists a partition $\left\lbrace V(G_1),S,V(G)-S-V(G_1) \right\rbrace $ of $V(G')$. The corresponding quotient matrix $M_3$ of $Q(G')$ has the form
	 \begin{equation*}
	 M_3=\begin{pmatrix}
	 2n_1+s-2 & s & 0\\
	 n_1 & n+s-2 & k-1 \\
	 0 & s & s
	 \end{pmatrix}.
	 \end{equation*}
	 
	  The characteristic polynomial $g(x)$ of the matrix $M_3$ can be calculated as below by expanding $M_3$'s determinant on the first line:
	 \begin{equation}
	 g(x)=(x-2n_1-s+2)[(x-n-s+2)(x-s)-s(k-1)]-n_1s(x-s).	 
	 \end{equation}
 
	 Let $r_g$ be the largest root of the equation $g(x)=0$. By Lemma \ref{equit}, we have $r_g>2n_1+s-2$, $r_g> n+s-2$ and $r_g>s$. Combining $g(r_g)=0$, yields $(r_g-n-s+2)(r_g-s)-s(k-1)>0$.
	 
	 Next, we construct a new graph $G'''$ obtained from $G'$ deleting the two components $G_k$ and $G_{k-1}$, which are single vertices, and adding two vertices to $G_1$ by joining the two vertices to the vertices in $V(G_1)$ and $S$.
	One can see that $G'''-S$ only has $k-2$ components and $n''_1=|V(G''_1)|=n_1+2$. Consider the partition $\left\lbrace V(G''),S,V(G)-S-V(G_1)\right\rbrace $, the characteristic polynomial of the quotient matrix of signless Laplacian matrix $Q(G''')$ of graph $G'''$ equals
	 \begin{align*}
	 \tilde{g}(x)&=(x-2n_1-s-2)[(x-n-s+2)(x-s)-s(k-3)]-(n_1+2)s(x-s)\\
	 &=g(x)-4sn_1-4s-4[(x-n-s+2)(x-s)-s(k-1)].
	 \end{align*}
	 
	 Similarly, plugging the value of $r_g$ into $\tilde{g}(x)$ yields
	 \begin{equation*}
	\tilde{g}(r_g)=-4sn_1-4s-4[(r_g-n-s+2)(r_g-s)-s(k-1)]<0,
	 \end{equation*}
	 which implies that $r_g<r_{\tilde{g}}$, where $r_{\tilde{g}}$ is the largest root of the equation $\tilde{g}(x)=0$. Since these partitions are equitable, we have $q_1(G')=r_g$ and $q_1(G''')=r_{\tilde{g}}$. Hence, $q_1(G')<q_1(G''')$ and $G'$ has the maximum spectral radius if $k=s+2$, i.e., $n_1=n-2s-1$. Similarly, we suppose that $|E(G')|=\dbinom{s+n_1}{2}+s(s+1)$, where $n_1=n-2s-1$. The quotient matrix of $Q(G')$ according to the partition $\left\lbrace V(G_1),S,V(G)-S-V(G_1) \right\rbrace $ is
	 \begin{equation*}
	 M_4=\begin{pmatrix}
	 2n_1+s-2 & s & 0\\
	 n_1 & n+s-2 & s+1 \\
	 0 & s & s
	 \end{pmatrix}.
	 \end{equation*}
	 
	 The characteristic polynomial of $M_4$ equals
	 \begin{align*}
	 h(x)&=(x-2n_1-s+2)[(x-n-s+2)(x-s)-s(s+1)]-n_1s(x-s)\\
	 &=(x-2n+3s+4)[(x-n-s+2)(x-s)-s(s+1)]-(n-2s-1)s(x-s)\\
	 &=x^3+(s-3n+6)x^2+(2n^2+ns-8n-4s^2-4s+8)x-2s(n^2-2ns-5n+s^2+5s+6).
	 \end{align*}
	 
	 Next, we discuss the difference between $q_1(G')$ and $q_1(K_{n-3}\vee K_1 \vee \overline{K_2})$. For $n\geq4$, one can easily check that $q_1(K_{n-3}\vee K_1 \vee \overline{K_2})=r(n)$. Furthermore, we have
	 \begin{align}
	 r(n)&=n+\frac{2^{\frac{2}{3}}\cdot\left(63n+3\sqrt{3}\cdot\sqrt{-4n^6+84n^5-781n^4+4074n^3-12633n^2+2232n-17376}-9n^2-38\right)^{\frac{1}{3}} }{6}\nonumber\\
	 &~~~~+\frac{2^{\frac{1}{3}}\cdot\left(3n^2-21n+49 \right) }{3\left(63n+3\sqrt{3}\cdot\sqrt{-4n^6+84n^5-781n^4+4074n^3-12633n^2+2232n-17376}-9n^2-38 \right)^{\frac{1}{3}} }-\frac{7}{3}.\label{root1}
	 \end{align}
	 
	 Plugging the value $r(n)$ into $h(x)$ yields
	 \begin{align*}
	 h(r(n))&=(s-1)r^2(n)+(ns-n-4s^2-4s+8)r(n)-2(sn^2-2ns^2-5ns+s^3+5s^2+6s-n^2+7n-12)\\
	 &=(s-1)\left[r^2(n)+(n-4s-8)r(n)-2(n^2-2ns-7n+s^2+6s+12) \right]. 
	 \end{align*}
	 
	 Note that $n\geq 2s+4$ (or $n_1\geq3$), then $h(r(n))\geq(s-1)\left[r^2(n)-(2s+4)r(n)-2s^2 \right] $. Together with (\ref{root1}), we obtain that $r^2(n)-(2s+4)r(n)-2s^2\geq4.2843$ with equality holding if and only if $s=1$ and $n_1=3$, which implies that $h(r(n))\geq0$.
	 
	 Finally, let $n_1=1$, i.e., $n=2s+2$. The quotient matrix of $Q(G')$ according to the partition $\left\lbrace S, V(G')-S \right\rbrace $ is
	 \begin{equation*}
	 M_5=\begin{pmatrix}
	 	n+s-2 & s+2 \\
	 	s     & s
	 \end{pmatrix}.
	 \end{equation*}
	 
	 The characteristic polynomial of $M_5$ equals
	 \begin{equation*}
	 l(x)=x^2+(2-2s-n)x+(sn-4s).
	 \end{equation*}
 	
 	Let $r_l$ be the largest root of the equation $l(x)=0$. By calculation, we have
	 \begin{equation*}
	 r_l=\frac{n+2s-2+\sqrt{n^2-4n+4s^2+8s+4}}{2}=\frac{2n-4+\sqrt{2n(n-2)}}{2}.
	 \end{equation*} 
	 
	 From the above consideration, it's not difficult to find the relation between $r(n)$ and $r_{l}$: (\romannumeral1)  For $n\geq10$ (or $s\geq3$), then $r(n)>r_l$; (\romannumeral2) For $n=6$ and $n=8$, (or $s=2$ and $s=3$), then $r(n)<r_l$; (\romannumeral3) For $n=4$ (or $s=1$), then $r(n)=r_l$.
\end{proof}

\section{Extremal graphs}
%

Finally, we  determine the signless Laplacian spectral radius of $K_{n-3}\vee K_1 \vee \overline{K_2}$, $K_2\vee\overline{K_4}$ and $K_3\vee\overline{K_5}$ to show that the bound on the signless Laplacian spectral radius shown in Theorem \ref{them1} is sharp. In addition, the graph $K_{n-3}\vee K_1 \vee \overline{K_2}$, $K_2\vee\overline{K_4}$ also can be used to prove that the bound given by Zhao et al. \cite{Yan} is sharp.
\begin{theorem}
	Let $r(n)$ be the largest root of the equation $x^3-(3n-7)x^2+n(2n-7)x-2(n^2-7n+12)=0$. Then $q_1(K_{n-3}\vee K_1 \vee \overline{K_2})=r(n)$ holds for any even positive integer $n\geq4$. In addition, we have $q_1(K_2\vee\overline{K_4})=4+2\sqrt{3}$ and $q_1(K_3\vee\overline{K_5})=6+2\sqrt{6}$.
\end{theorem}
\begin{proof}
	Consider the vertex partition $\left\lbrace V(K_{n-3}),V(K_1),V(\overline{K_2}) \right\rbrace $ of the graph $K_{n-3}\vee K_1 \vee \overline{K_2}$. The corresponding quotient matrix of $Q(K_{n-3}\vee K_1 \vee \overline{K_2})$ equals
	\begin{equation*}
	\begin{pmatrix}
		2n-7 & 1   & 0 \\
		n-3  & n-1 & 2 \\
		0    & 1   & 1
	\end{pmatrix}.
	\end{equation*}
	
	It's easy to calculate that the characteristic polynomial of the above matrix is $x^3-(3n-7)x^2+n(2n-7)x-2(n^2-7n+12)=0$. Since the above quotient matrix is equitable, then we have  $q_1(K_{n-3}\vee K_1 \vee \overline{K_2})=r(n)$.
	
	Analogously, for the vertex partition $\left\lbrace V(K_2),V(\overline{K_4}) \right\rbrace $ of the graph $K_2\vee\overline{K_4}$ and $\left\lbrace V(K_3),V(\overline{K_5}) \right\rbrace $ of the graph $K_3\vee\overline{K_5}$, the corresponding quotient matrix of $Q(K_2\vee\overline{K_4})$ and $Q(K_3\vee\overline{K_5})$ equal

	\begin{equation*}
	\begin{pmatrix}
	6 & 4 \\
	2 & 2
	\end{pmatrix},
	\end{equation*}
	and
	\begin{equation*}
	\begin{pmatrix}
	9 & 5 \\
	3 & 3
	\end{pmatrix}.
	\end{equation*}
	
	Since these partitions are equitable, calculating their spectral radius yields that $q_1(K_2\vee\overline{K_4})=4+2\sqrt{3}$ and $q_1(K_3\vee\overline{K_5})=6+2\sqrt{6}$.
\end{proof}

Note that $q_1(K_2\vee\overline{K_4})=4+2\sqrt{3}\approx7.4641>6.9095=r(6)$ and $q_1(K_3\vee\overline{K_5})=6+2\sqrt{6}\approx10.8990>10.5136=r(8)$. This verifies the conclusions in Theorem \ref{them1}.


\section*{Acknowledgements} The authors would like to express their sincere gratitude to all the referees for their careful reading and insightful suggestions.

\end{document}